\documentclass[12pt,a4paper]{amsart}

\usepackage[hmargin=2cm,vmargin=2.5cm]{geometry}

\usepackage{amsmath,amssymb}
\usepackage{tikz}

\def\mcM{{\mathcal{M}}}
\def\mcY{{\mathcal Y}}
\def\mfM{\mathfrak M}

\def\Z{{\mathbb Z}}
\def\mbZ{{\mathbb Z}}
\def\mbC{{\mathbb C}}
\def\mbC{{\mathbb C}}
\def\mbL{{\mathbb L}}

\DeclareMathOperator{\Hom}{Hom}
\DeclareMathOperator{\im}{Im}

\def\mmod{{\text{ mod }}}

\newtheorem{proposition}{Proposition}[section]
\newtheorem{theorem}[proposition]{Theorem}

\newtheorem{lemma}[proposition]{Lemma}

\newtheorem{remark}[proposition]{Remark}

\title[Quasihomogeneous Hilbert schemes]{Generating series of the Poincar\'e polynomials of quasihomogeneous Hilbert schemes}

\author{A. Buryak} 

\address{A.~Buryak:\newline
Department of Mathematics,
University of Amsterdam, \newline
P.~O.~Box 94248, 1090 GE Amsterdam, 
The Netherlands\newline 
\indent and\newline
Department of Mathematics, Moscow State University,\newline
Leninskie gory, 119992 GSP-2 Moscow, Russia} 
\email{a.y.buryak@uva.nl, buryaksh@mail.ru}

\author{B. L. Feigin}

\address{B.~L.~Feigin:\newline
National research university Higher school of Economics, Russia, Moscow, 101000, Myasnitskaya ul., 20, \newline
Landau Institute for Theoretical Physics, Russia, Chernogolovka, 142432, prosp. Akademika Semenova, 1a, and \newline
Independent University of Moscow, Russia, Moscow, 119002, Bolshoy Vlasyevskiy per., 11.}
\email{borfeigin@gmail.com}

\begin{document}

\begin{abstract}
In this paper we prove that the generating series of the Poincar\'e polynomials of quasihomogeneous Hilbert schemes of points in the plane has a beautiful decomposition into an infinite product. We also compute the generating series of the numbers of quasihomogeneous components in a moduli space of sheaves on the projective plane. The answer is given in terms of characters of the affine Lie algebra~$\widehat{sl}_m$.
\end{abstract}

\maketitle

\section{Introduction}

The Hilbert scheme $(\mbC^2)^{[n]}$ of $n$ points in the plane $\mbC^2$ parametrizes ideals $I\subset\mbC[x,y]$ of colength $n$: $\dim_{\mbC}\mbC[x, y]/I=n$. There is an open dense subset of $(\mbC^2)^{[n]}$, that parametrizes the ideals, associated with configurations of $n$ distinct points. The Hilbert scheme of $n$ points in the plane is a nonsingular, irreducible, quasiprojective algebraic variety of dimension $2n$ with a rich and much studied geometry, see \cite{Gottsche,Nakajima} for an introduction.

The cohomology groups of $(\mbC^2)^{[n]}$ were computed in \cite{Ellingsrud} and we refer the reader to the papers \cite{Costello,Lehn1,Lehn2,Li,Vasserot} for the description of the ring structure in the cohomology $H^*((\mbC^2)^{[n]})$. 

There is a $(\mbC^*)^2$-action on $(\mbC^2)^{[n]}$ that plays a central role in this subject. The algebraic torus $(\mbC^*)^2$ acts on $\mbC^2$ by scaling the coordinates, $(t_1,t_2)\cdot(x,y)=(t_1x,t_2y)$. This action lifts to the $(\mbC^*)^2$-action on the Hilbert scheme $(\mbC^2)^{[n]}$.

Let $T_{\alpha,\beta}=\{(t^\alpha,t^\beta)\in(\mbC^*)^2|t\in\mbC^*\}$, where $\alpha,\beta\ge 1$ and $gcd(\alpha,\beta)=1$, be a one dimensional subtorus of $(\mbC^*)^2$. The variety $\left((\mbC^2)^{[n]}\right)^{T_{\alpha,\beta}}$ parametrizes quasihomogeneous ideals of colength $n$ in the ring $\mbC[x,y]$. Irreducible components of $\left((\mbC^2)^{[n]}\right)^{T_{\alpha,\beta}}$ were described in \cite{Evain}. Poincar\'e polynomials of irreducible components in the case $\alpha=1$ were computed in~\cite{Buryak}. For $\alpha=\beta=1$ it was done in~\cite{Iarrobino}.

For a manifold $X$ let $H_*(X)$ denote the homology group of $X$ with rational coefficients. Let $P_q(X)=\sum_{i\ge 0}\dim H_i(X)q^{\frac{i}{2}}$. The main result of this paper is the following theorem (it was conjectured in \cite{Buryak}):
\begin{theorem}\label{theorem:quasihomogeneous}
\begin{gather}\label{formula:quasihomogeneous}
\sum_{n\ge 0}P_q\left(\left((\mbC^2)^{[n]}\right)^{T_{\alpha,\beta}}\right)t^n=\prod_{\substack{i\ge 1\\(\alpha+\beta)\nmid i}}\frac{1}{1-t^i}\prod_{i\ge 1}\frac{1}{1-qt^{(\alpha+\beta)i}}.
\end{gather}
\end{theorem}

There is a standard method for constructing a cell decomposition of the Hilbert scheme $\left((\mbC^2)^{[n]}\right)^{T_{\alpha,\beta}}$ using the Bialynicki-Birula theorem. In this way the Poincar\'e polynomial of this Hilbert scheme can be written as a generating function for a certain statistic on Young diagrams of size $n$. However, it happens that this combinatorial approach doesn't help in a proof of Theorem~\ref{theorem:quasihomogeneous}. In fact, we get very nontrivial combinatorial identities as a corollary of this theorem, see Section \ref{subsection:combinatorial identities}.

We can describe the main geometric idea in the proof of Theorem~\ref{theorem:quasihomogeneous} in the following way. The irreducible components of $\left((\mbC^2)^{[n]}\right)^{T_{\alpha,\beta}}$ can be realized as fixed point sets of a $\mbC^*$-action on cyclic quiver varieties. Theorem~\ref{firsttheorem} tells us that the Betti numbers of the fixed point set are equal to the shifted Betti numbers of the quiver variety. Then known results about cohomology of quiver varieties can be used for a proof of Theorem~\ref{theorem:quasihomogeneous}.

In principle, Theorem~\ref{firsttheorem} has an independent interest. However, there is another application of this theorem. In~\cite{BuryakFeigin} we studied the generating series of the numbers of quasihomogeneous components in a moduli space of sheaves on the projective plane. Combinatorially we managed to compute it only in the simplest case. Now using Theorem~\ref{firsttheorem} we can give an answer in a general case, this is Theorem~\ref{theorem:irreducible components}. We show that it proves our conjecture from~\cite{BuryakFeigin}.

\subsection{Combinatorial identities}\label{subsection:combinatorial identities}

Here we formulate two combinatorial identities that follow from Theorem \ref{theorem:quasihomogeneous}. We denote by $\mcY$ the set of all Young diagrams. For a Young diagram $Y$ let 
\begin{align*}
&r_l(Y)=|\{(i,j)\in Y|j=l\}|,\\
&c_l(Y)=|\{(i,j)\in Y|i=l\}|.
\end{align*}
For a point $s=(i,j)\in\Z_{\ge 0}^2$ let
\begin{align*}
&l_Y(s)=r_j(Y)-i-1,\\
&a_Y(s)=c_i(Y)-j-1,
\end{align*}
see Figure \ref{armlegpic}. Note that $l_Y(s)$ and $a_Y(s)$ are negative, if $s\notin Y$.

\begin{figure}[h]
\begin{center}
\begin{picture}(70,40)
\put(0,6){
\multiput(0,0)(0,5){7}{\line(1,0){5}}
\multiput(5,0)(0,5){6}{\line(1,0){5}}
\multiput(10,0)(0,5){6}{\line(1,0){5}}
\multiput(15,0)(0,5){4}{\line(1,0){5}}

\multiput(0,0)(5,0){5}{\line(0,1){5}}
\multiput(0,5)(5,0){5}{\line(0,1){5}}
\multiput(0,10)(5,0){5}{\line(0,1){5}}
\multiput(0,15)(5,0){4}{\line(0,1){5}}
\multiput(0,20)(5,0){4}{\line(0,1){5}}
\multiput(0,25)(5,0){2}{\line(0,1){5}}

\put(6.5,6.5){$s$}
\multiput(11,6)(5,0){2}{$\spadesuit$}
\multiput(5.8,10.8)(0,5){3}{$\heartsuit$}

\put(31,11){$a_Y(s)=$ number of $\heartsuit$}
\put(31,16){$l_Y(s)=$ number of $\spadesuit$}
}
\put(8.5,0){$Y$}
\end{picture}
\end{center}
\caption{}
\label{armlegpic}
\end{figure}
The number of boxes in a Young diagram $Y$ is denoted by $|Y|$.
\begin{theorem}\label{first identity}
Let $\alpha$ and $\beta$ be two arbitrary positive coprime integers. Then we have
\begin{gather*}\label{formula:first identity}
\sum_{Y\in\mcY}q^{\sharp\{s\in Y|\alpha l(s)=\beta(a(s)+1)\}}t^{|Y|}=\prod_{\substack{i\ge 1\\(\alpha+\beta)\nmid i}}\frac{1}{1-t^i}\prod_{i\ge 1}\frac{1}{1-qt^{(\alpha+\beta)i}}.
\end{gather*}
\end{theorem}

In the case $\alpha=\beta=1$ another identity can be derived from Theorem~\ref{theorem:quasihomogeneous}. The $q$-binomial coefficients are defined by
$$
\genfrac[]{0pt}{}{M}{N}_q=\frac{\prod\limits_{i=1}^M(1-q^i)}{\prod\limits_{i=1}^N(1-q^i)\prod\limits_{i=1}^{M-N}(1-q^i)}.
$$
By $\mathcal P$ we denote the set of all partitions. For a partition $\lambda=(\lambda_1,\lambda_2,\ldots,\lambda_r), \lambda_1\ge\lambda_2\ge\ldots\ge\lambda_r$, let $|\lambda|=\sum_{i=1}^r\lambda_i$.
\begin{theorem}\label{second identity}
$$
\sum_{\lambda\in\mathcal P}\prod_{i\ge 1}\genfrac[]{0pt}{}{\lambda_i-\lambda_{i+2}+1}{\lambda_{i+1}-\lambda_{i+2}}_qt^{\frac{\lambda_1(\lambda_1-1)}{2}+|\lambda|}=\prod_{i\ge 1}\frac{1}{(1-t^{2i-1})(1-qt^{2i})}.
$$
\end{theorem}
Here for a partition $\lambda=(\lambda_1,\lambda_2,\ldots,\lambda_r), \lambda_1\ge\lambda_2\ge\ldots\ge\lambda_r$, we adopt the convention $\lambda_{>r}=0$.

\subsection{Cyclic quiver varieties}

Quiver varieties were introduced by H.~Nakajima in \cite{Nakajima1}. Here we review the construction in the particular case of cyclic quiver varieties. We follow the approach from \cite{Nakajima3}.  

Let $m\ge 2$. We fix vector spaces $V_0,V_1,\ldots,V_{m-1}$ and $W_0,W_1,\ldots,W_{m-1}$ and we denote by
$$
v=(\dim V_0,\ldots,\dim V_{m-1}), w=(\dim W_0,\ldots,\dim W_{m-1})\in \mbZ_{\ge 0}^m.
$$
the dimension vectors. We adopt the convention $V_m=V_0$. Let 
\begin{align*}
M(v,w)=&\left(\bigoplus_{k=0}^{m-1}\Hom(V_k,V_{k+1})\right)\oplus\left(\bigoplus_{k=0}^{m-1}\Hom(V_k,V_{k-1})\right)\\
&\oplus\left(\bigoplus_{k=0}^{m-1}\Hom(W_k,V_k)\right)\oplus\left(\bigoplus_{k=0}^{m-1}\Hom(V_k,W_k)\right).
\end{align*}
The group $G_v=\prod\limits_{k=0}^{m-1}GL(V_k)$ acts on $M(v,w)$ by
$$
g\cdot (B_1,B_2,i,j)\mapsto(g B_1 g^{-1},g B_2 g^{-1},g i,jg^{-1}).
$$
The map $\mu\colon M(v,w)\to \bigoplus\limits_{k=0}^{m-1}\Hom(V_k,V_k)$ is defined as follows
$$
\mu(B_1,B_2,i,j)=[B_1,B_2]+ij.
$$ 
Let 
$$
\mu^{-1}(0)^s=\left\{(B,i,j)\in\mu^{-1}(0)\left|
\begin{smallmatrix}
\text{if a collection of subspaces $S_k\subset V_k$}\\
\text{is $B$-invariant and contains $\im(i)$, then $S_k=V_k$}
\end{smallmatrix}
\right.\right\}.
$$
The action of $G_v$ on $\mu^{-1}(0)^s$ is free. The quiver variety $\mfM(v,w)$ is defined as the quotient 
$$
\mfM(v,w)=\mu^{-1}(0)^s/G_v,
$$
see Figure \ref{pic1}. 

The variety $\mfM(v,w)$ is irreducible (see e.g.\cite{Nakajima3}). 

We define the $(\mbC^*)^2\times(\mbC^*)^m$-action on $\mfM(v,w)$ as follows:
$$
(t_1,t_2,e_k)\cdot (B_1,B_2,i_k,j_k)=(t_1 B_1,t_2 B_2,e_k^{-1}i_k,t_1t_2e_k j_k).
$$

\begin{figure}[t]

\begin{tikzpicture}[auto]
\node (a) at (90:2.5) {$V_0$};
\node (e) at (162:2.5) {$V_{m-1}$};
\node (d) at (234:2.5) {$V_{m-2}$};
\node (c) at (306:2.5) {$V_2$};
\node (b) at (18:2.5) {$V_1$};

\node (aa) at (90:4.3) {$W_0$};
\node (ee) at (162:4.3) {$W_{m-1}$};
\node (dd) at (234:4.3) {$W_{m-2}$};
\node (cc) at (306:4.3) {$W_2$};
\node (bb) at (18:4.3) {$W_1$};

\draw [->] (a) to [bend left=12] node {$B_1$} (b);
\draw [->] (b) to [bend left=12] node {$B_2$} (a);
\draw [->] (b) to [bend left=12] node {$B_1$} (c);
\draw [->] (c) to [bend left=12] node {$B_2$} (b);
\draw [dashed] (c) to (d);
\draw [->] (d) to [bend left=12] node {$B_1$} (e);
\draw [->] (e) to [bend left=12] node {$B_2$} (d);
\draw [->] (e) to [bend left=12] node {$B_1$} (a);
\draw [->] (a) to [bend left=12] node {$B_2$} (e);

\draw [->] (a) to [bend left=12] node {$j$} (aa);
\draw [->] (aa) to [bend left=12] node {$i$} (a);
\draw [->] (b) to [bend left=12] node {$j$} (bb);
\draw [->] (bb) to [bend left=12] node {$i$} (b);
\draw [->] (c) to [bend left=12] node {$j$} (cc);
\draw [->] (cc) to [bend left=12] node {$i$} (c);
\draw [->] (d) to [bend left=12] node {$j$} (dd);
\draw [->] (dd) to [bend left=12] node {$i$} (d);
\draw [->] (e) to [bend left=12] node {$j$} (ee);
\draw [->] (ee) to [bend left=12] node {$i$} (e);
\end{tikzpicture}
\caption{Cyclic quiver variety $\mfM(v,w)$}

\label{pic1}
\end{figure}

\subsection{$\mbC^*$-action on $\mfM(v,w)$}

In this section we formulate Theorem~\ref{firsttheorem} that is a key step in the proofs of Theorems \ref{theorem:quasihomogeneous} and \ref{theorem:irreducible components}.
 
Let $\alpha$ and $\beta$ be any two positive coprime integers, such that $\alpha+\beta=m$. Define the integers $\lambda_0,\lambda_1,\ldots,\lambda_{m-1}\in[-(m-1),0]$ by the formula $\lambda_k\equiv-\alpha k\mmod m$. We define the one-dimensional subtorus $\widetilde T_{\alpha,\beta}\subset(\mbC^*)^2\times(\mbC^*)^m$ by
$$
\widetilde T_{\alpha,\beta}=\{(t^\alpha,t^\beta,t^{\lambda_0},t^{\lambda_1},\ldots,t^{\lambda_{m-1}})\in(\mbC^*)^2\times (\mbC^*)^m|t\in\mbC^*\}.
$$
For a manifold $X$ we denote by $H_*^{BM}(X)$ the homology group of possibly infinite singular chains with locally finite support (the Borel-Moore homology) with rational coefficients. Let $P^{BM}_q(X)=\sum\limits_{i\ge 0}\dim H_i^{BM}(X)q^{\frac{i}{2}}$.
\begin{theorem}\label{firsttheorem}
The fixed point set $\mfM(v,w)^{\widetilde T_{\alpha,\beta}}$ is compact and 
$$
P^{BM}_q(\mfM(v,w))=q^{\frac{1}{2}\dim\mfM(v,w)}P_q\left(\mfM(v,w)^{\widetilde T_{\alpha,\beta}}\right).
$$
\end{theorem}

\subsection{Quasihomogeneous components in the moduli space of sheaves}

Here we formulate our result that relates the numbers of quasihomogeneous components in a moduli space of sheaves with characters of the affine Lie algebra $\widehat{sl}_m$.

The moduli space $\mcM(r,n)$ is defined as follows (see e.g.\cite{Nakajima}):
\begin{gather*}
\mcM(r,n)=\left.\left\{(B_1,B_2,i,j)\left|
\begin{smallmatrix}
1) [B_1,B_2]+ij=0\\
2) \text{(stability) There is no subspace} \\
\text{$S\subsetneq\mbC^n$ such that $B_{\alpha}(S)\subset S$ ($\alpha=1,2$)}\\
\text{and $\mathop{Im}(i)\subset S$} 
\end{smallmatrix}\right.\right\}\right/GL_n(\mbC),
\end{gather*}
where $B_1,B_2\in End(\mbC^n), i\in Hom(\mbC^r,\mbC^n)$ and $j\in Hom(\mbC^n,\mbC^r)$ with the action of $GL_n(\mbC)$ given by 
\begin{gather*}
g\cdot(B_1,B_2,i,j)=(gB_1g^{-1},gB_2g^{-1},gi,jg^{-1}),
\end{gather*} 
for $g\in GL_n(\mbC)$. 

The variety $\mcM(r,n)$ has another description as the moduli space of framed torsion free sheaves on the projective plane, but for our purposes the given definition is better. We refer the reader to \cite{Nakajima} for details. The variety $\mcM(1,n)$ is isomorphic to $(\mbC^2)^{[n]}$ (see e.g.\cite{Nakajima}). 

Define the $(\mbC^*)^2\times(\mbC^*)^r$-action on $\mcM(r,n)$ by
\begin{gather*}
(t_1,t_2,e)\cdot[(B_1,B_2,i,j)]=[(t_1B_1,t_2B_2,ie^{-1},t_1t_2ej)].
\end{gather*}

Consider two positive coprime integers $\alpha$ and $\beta$ and a vector 
$$
\vec\omega=(\omega_1,\omega_2,\ldots,\omega_r)\in\mbZ^r
$$
such that $0\le\omega_i< \alpha+\beta$. Let $T_{\alpha,\beta}^{\vec\omega}$ be the one-dimensional subtorus of $(\mbC^*)^2\times(\mbC^*)^r$ defined by 
$$
T_{\alpha,\beta}^{\vec\omega}=\{(t^{\alpha},t^{\beta},t^{\omega_1},t^{\omega_2},\ldots,t^{\omega_r})\in (\mbC^*)^2\times(\mbC^*)^r|t\in\mbC^*\}.
$$
In \cite{BuryakFeigin} we studied the numbers of the irreducible components of $\mcM(r,n)^{T_{\alpha,\beta}^{\vec\omega}}$ and found an answer in the case $\alpha=\beta=1$. Now we can solve the general case.

We define the vector $\vec\rho=(\rho_0,\rho_1,\ldots,\rho_{\alpha+\beta-1})\in\mbZ_{\ge 0}^{\alpha+\beta}$ by $\rho_i=\sharp\{j|\omega_j=i\}$ and the vector $\vec\mu\in\mbZ_{\ge 0}^{\alpha+\beta}$ by 
$\mu_i=\rho_{-i\alpha\mmod\alpha+\beta}$. 

Let $E_k,F_k,H_k$, $k=1,2,\ldots,\alpha+\beta$, be the standard generators of $\widehat{sl}_{\alpha+\beta}$. Let $\mathcal V$ be the irreducible highest weight representation of $\widehat{sl}_{\alpha+\beta}$ with the highest weight $\vec\mu$. Let $x\in\mathcal V$ be the highest weight vector. We denote by $\mathcal V_p$ the vector subspace of $\mathcal V$ generated by vectors~$F_{i_1}F_{i_2}\ldots F_{i_p}x$. The character $\chi_{\vec\mu}(q)$ is defined by
$$
\chi_{\vec\mu}(q)=\sum_{p\ge 0}(\dim\mathcal V_p) q^p.
$$ 
We denote by $h_0(X)$ the number of connected components of a manifold~$X$. 

\begin{theorem}\label{theorem:irreducible components}
$$
\sum_{n\ge 0}h_0\left(\mcM(r,n)^{T^{\vec\omega}_{\alpha,\beta}}\right)q^n=\chi_{\vec\mu}(q).
$$
\end{theorem}
In \cite{Jimbo} the authors found a combinatorial formula for characters of $\widehat{sl}_m$ in terms of Young diagrams with certain restrictions. In \cite{FFJMM} the same combinatorics is used to give a formula for certain characters of the quantum continuous $gl_\infty$. Comparing these two combinatorial formulas it is easy to see that Conjecture 1.2 from~\cite{BuryakFeigin} follows from Theorem \ref{theorem:irreducible components}. 
\begin{remark}
There is a small mistake in Conjecture 1.2 from~\cite{BuryakFeigin}. The vector $\vec {a'}=(a_0',a_1',\ldots,a_{\alpha+\beta-1}')$ should be defined by $a'_i=a_{-\alpha i\mmod \alpha+\beta}$. The rest is correct.
\end{remark}

\subsection{Organization of the paper} 

We prove Theorem~\ref{firsttheorem} in Section~\ref{section:first theorem}. Then using this result we prove Theorem~\ref{theorem:quasihomogeneous} in~Section~\ref{section:quasihomogeneous}. In Section~\ref{section:combinatorial identities} we derive the combinatorial identities as a corollary of Theorem~\ref{theorem:quasihomogeneous}. Finally, using Theorem~\ref{firsttheorem} we prove Theorem~\ref{theorem:irreducible components} in Section~\ref{section:irreducible components}. 

\subsection{Acknowledgments}

The authors are grateful to S. M. Gusein-Zade, M. Finkelberg and S. Shadrin for useful discussions. 

A. B. is partially supported by a Vidi grant of the Netherlands Organization of Scientific Research, by the grants RFBR-10-01-00678, NSh-4850.2012.1 and the Moebius Contest Foundation for Young Scientists. Research of B. F. is partially supported by RFBR initiative interdisciplinary project grant 09-02-12446-ofi-m, by RFBR-CNRS grant 09-02-93106, RFBR grants 08-01-00720-a, NSh-3472.2008.2 and 07-01-92214-CNRSL-a.

\section{Proof of Theorem \ref{firsttheorem}}\label{section:first theorem}

In this section we prove Theorem~\ref{firsttheorem}. The Grothendieck ring of quasiprojective varieties is a useful technical tool and we remind its definition and necessary properties in Section~\ref{subsection:Grothendieck}.     

\subsection{Grothendieck ring of quasiprojective varieties} \label{subsection:Grothendieck}

The Grothendieck ring $K_0(\nu_{\mbC})$ of complex quasiprojective varieties is the abelian group generated by the classes $[X]$ of all complex quasiprojective varieties $X$ modulo the relations:
\begin{enumerate}
\item if varieties $X$ and $Y$ are isomorphic, then $[X]=[Y]$;
\item if $Y$ is a Zariski closed subvariety of $X$, then $[X]=[Y]+[X\backslash Y]$.
\end{enumerate}  
The multiplication in $K_0(\nu_{\mbC})$ is defined by the Cartesian product of varieties: $[X_1]\cdot[X_2]=[X_1\times X_2]$. The class $\left[\mathbb A^1_{\mbC}\right]\in K_0(\nu_{\mbC})$ of the complex affine line is denoted by $\mbL$.

We need the following property of the ring $K_0(\nu_{\mbC})$. There is a natural homomorphism of rings $\theta\colon\mbZ[z]\to K_0(\nu_{\mbC})$, defined by $\theta(z)=\mbL$. This homorphism is an inclusion (see e.g.\cite{Looijenga}).

\subsection{Proof of Theorem \ref{firsttheorem}} 

Let $r=\sum_{i=0}^{m-1}\theta_i$. For an arbitrary $\vec\nu\in\mbZ^r$ let $\Gamma^{\vec\nu}_{\alpha,\beta}\subset T_{\alpha,\beta}^{\vec\nu}$ be the subgroup of roots of $1$ of degree $m$.  
Let
$$
\vec\theta=(\underbrace{0,\ldots,0}_{\text{$w_0$ times}},\underbrace{\lambda_1,\ldots,\lambda_1}_{\text{$w_1$ times}},\ldots,\underbrace{\lambda_{m-1},\ldots,\lambda_{m-1}}_{\text{$w_{m-1}$ times}})\in\mbZ^{r}
$$
\begin{lemma}\label{lemma:finite group}
1. We have the following decomposition into irreducible components
\begin{gather}\label{equation:lemma1}
\mcM(r,n)^{\Gamma_{\alpha,\beta}^{\vec\theta}}=\coprod_{\substack{v\in\Z_{\ge 0}^m\\\sum v_k=n}}\mfM(v,w).
\end{gather}
2.The $T^{\vec\theta}_{\alpha,\beta}$-action on the left-hand side of~\eqref{equation:lemma1} corresponds to the $\widetilde T_{\alpha,\beta}$-action on the right-hand side of~\eqref{equation:lemma1}.
\end{lemma}
\begin{proof}
Let $\Gamma_m$ be the group of roots of unity of degree $m$. By definition, a point $[(B_1,B_2,i,j)]\in\mcM(r,n)$ is fixed under the action of ${\Gamma^{\vec\theta}_{\alpha,\beta}}$ if and only if there exists a homomorphism $\lambda\colon\Gamma_m\to GL_n(\mbC)$ satisfying the following conditions: 
\begin{align}\label{formula:fixed point}
\zeta^{\alpha}B_1&=\lambda(\zeta)^{-1}B_1\lambda(\zeta),\notag\\
\zeta^{\beta}B_2&=\lambda(\zeta)^{-1}B_2\lambda(\zeta),\\
i\circ diag(\zeta^{\theta_1},\zeta^{\theta_2},\ldots,\zeta^{\theta_r})^{-1}&=\lambda(\zeta)^{-1}i,\notag\\
diag(\zeta^{\theta_1},\zeta^{\theta_2},\ldots,\zeta^{\theta_r})\circ j&=j\lambda(\zeta),\notag
\end{align} 
where $\zeta=e^{\frac{2\pi\sqrt{-1}}{m}}$.
Suppose that $[(B_1,B_2,i,j)]$ is a fixed point. Then we have the weight decomposition of $\mbC^n$ with respect to $\lambda(\zeta)$, i.e. $\mbC^n=\bigoplus_{k\in\mbZ/m\mbZ} V'_k$, where $V'_k=\{v\in \mbC^n|\lambda(\zeta)\cdot v=\zeta^kv\}$. We also have the weight decomposition of $\mbC^r$, i.e. $\mbC^r=\bigoplus_{k\in\mbZ/m\mbZ} W'_k$, where $W'_k=\{v\in \mbC^r|diag(\zeta^{\theta_1},\ldots,\zeta^{\theta_r})\cdot v=\zeta^kv\}$. From conditions \eqref{formula:fixed point} it follows that the only components of $B_1$, $B_2$, $i$ and $j$ that might survive are: 
\begin{align*}
B_1&\colon V'_k\to V'_{k-\alpha},\\
B_2&\colon V'_k\to V'_{k-\beta},\\
i&\colon W'_k\to V'_k,\\
j&\colon V'_k\to W'_k.
\end{align*}
Let us denote $V'_{-\alpha k\mmod m}$ by $V_k$ and $W'_{-\alpha k\mmod m}$ by $W_k$. Then the operators $B_1,B_2,i,j$ act as follows: $B_{1,2}\colon V_k\to V_{k\pm 1}, i\colon W_k\to V_k, j\colon V_k\to W_k$. The first part of the lemma is proved. The second part of the lemma easily follows from the proof of the first part and from the definition of the $\widetilde T_{\alpha,\beta}$-action.  
\end{proof}

In \cite{BuryakFeigin} it is proved that the variety $\mcM(r,n)^{T^{\vec\theta}_{\alpha,\beta}}$ is compact. Therefore, $\mfM(v,w)^{\widetilde T_{\alpha,\beta}}$ is compact.

We denote by $\mcM(r,n)^{\Gamma^{\vec\theta}_{\alpha,\beta}}_v$ the irreducible component of $\mcM(r,n)^{\Gamma^{\vec\theta}_{\alpha,\beta}}$ corresponding to $\mfM(v,w)$. Let $\mcM(r,n)^{T^{\vec\theta}_{\alpha,\beta}}_v=\left(\mcM(r,n)^{\Gamma^{\vec\theta}_{\alpha,\beta}}_v\right)^{T^{\vec\theta}_{\alpha,\beta}}$. We denote by $I_v$ the set of irreducible components of $\mcM(r,n)^{T^{\vec\theta}_{\alpha,\beta}}_v$ and let $\mcM(r,n)^{T^{\vec\theta}_{\alpha,\beta}}_v=\coprod\limits_{i\in I_v}\mcM(r,n)^{T^{\vec\theta}_{\alpha,\beta}}_{v,i}$ be the decomposition into the irreducible components. We define the sets $C_{v,i}$ by 
$$
C_{v,i}=\left\{z\in\mcM(r,n)^{\Gamma^{\vec\theta}_{\alpha,\beta}}_v\left|\lim_{t\to 0,t\in T_{\alpha,\beta}^{\vec\theta}}tz\in \mcM(r,n)^{T^{\vec\theta}_{\alpha,\beta}}_{v,i}\right.\right\}.
$$
\begin{lemma}\label{lemma:decomposition}
1) The sets $C_{v,i}$ form a decomposition of $\mcM(r,n)^{\Gamma^{\vec\theta}_{\alpha,\beta}}_v$ into locally closed subvarieties.\\
2) The subvariety $C_{v,i}$ is a locally trivial bundle over $\mcM(r,n)^{T^{\vec\theta}_{\alpha,\beta}}_{v,i}$ with an affine space as a fiber.
\end{lemma}
\begin{proof}
The lemma follows from the results of \cite{B1,B2}. The only thing that we need to check is that the limit $\lim_{t\to 0,t\in T_{\alpha,\beta}^{\vec\theta}}tz$ exists for any $z\in\mcM(r,n)^{\Gamma^{\vec\theta}_{\alpha,\beta}}_v$. 

Consider the variety $\mcM_0(r,n)$ from \cite{Nakajima2}. It is defined as the affine algebro-geometric quotient 
$$
\mcM_0(r,n)=\{(B_1,B_2,i,j)|[B_1,B_2]+ij=0\}//GL_n(\mbC).
$$
It can be viewed as the set of closed orbits in $\{(B_1,B_2,i,j)|[B_1,B_2]+ij=0\}$. There is a morphism $\pi\colon\mcM(r,n)\to\mcM_0(r,n)$. It maps a point $[(B_1,B_2,i,j)]\in \mcM(r,n)$ to the unique closed orbit that is contained in the closure of the orbit of $(B_1,B_2,i,j)$ in $\{(B_1,B_2,i,j)|[B_1,B_2]+ij=0\}$. The $(\mbC^*)^2\times(\mbC^*)^r$-action on $\mcM_0(r,n)$ is defined in the same way as on $\mcM(r,n)$. The variety $\mcM_0(r,n)$ is affine and the morphism $\pi$ is projective and equivariant (see e.g.\cite{Nakajima2}).
 
By \cite{Lusztig}, the coordinate ring of $\mcM_0(r,n)$ is generated by the following two types of functions:
\begin{itemize}
\item[a)]
$tr(B_{a_N}B_{a_{N-1}}\cdots B_{a_1}\colon\mbC^n\to\mbC^n)$, where $a_i=1$ or $2$.
\item[b)]
$\chi(jB_{a_N}B_{a_{N-1}}\cdots B_{a_1}i)$, where $a_i=1$ or $2$, and $\chi$ is a linear form on $End(\mbC^r)$.
\end{itemize}

From the inequalities $-m<\theta_k\le 0$ it follows that both types of functions have positive weights with respect to the $T^{\vec\theta}_{\alpha,\beta}$-action. Therefore, for any point $z\in\mcM_0(r,n)$ we have $\lim_{t\to 0,t\in T^{\vec\theta}_{\alpha,\beta}}tz=0$. The morphism $\pi$ is projective, so  the limit $\lim_{t\to 0,t\in T_{\alpha,\beta}^{\vec\theta}}tz$ exists for any $z\in\mcM(r,n)^{\Gamma^{\vec\theta}_{\alpha,\beta}}_v$. The lemma is proved.
\end{proof}

Denote by $d_{v,i}^+$ the dimension of the fiber of the locally trivial bundle $C_{v,i}\to\mcM(r,n)^{T^{\vec\theta}_{\alpha,\beta}}_{v,i}$. 
\begin{lemma}\label{lemma:dimension}
The dimension $d_{v,k}^+$ doesn't depend on $k\in I_v$ and is equal to
$$
d_{v,k}^+=\frac{1}{2}\dim\mfM(v,w).
$$
\end{lemma}
\begin{proof}
The set of fixed points of the $(\mbC^*)^2\times(\mbC^*)^r$-action on $\mcM(r,n)$ is finite and is parametrized by the set of $r$-tuples $D=(D_1,D_2,\ldots,D_r)$ of Young diagrams $D_i$ such that $\sum_{i=1}^r|D_i|=n$ (see~e.g.\cite{Nakajima2}).  

Let $p\in\mcM(r,n)^{(\mbC^*)^2\times(\mbC^*)^r}$ be the fixed point corresponding to an $r$-tuple $D$. Let $R((\mbC^*)^2\times(\mbC^*)^r)=\Z[t_1^{\pm 1},t_2^{\pm 1},e_1^{\pm 1},e_2^{\pm 1},\ldots,e_r^{\pm 1}]$ be the representation ring of $(\mbC^*)^2\times(\mbC^*)^r$. Then the weight decomposition of the tangent space $T_p\mcM(r,n)$ of the variety $\mcM(r,n)$ at the point $p$ is given by~(see~e.g.\cite{Nakajima2})
\begin{gather}\label{weight decomposition}
T_p\mcM(r,n)=\sum_{i,j=1}^r e_j e_i^{-1}\left(\sum_{s\in D_i}t_1^{-l_{D_j}(s)}t_2^{a_{D_i}(s)+1}+\sum_{s\in D_j}t_1^{l_{D_i}(s)+1}t_2^{-a_{D_j}(s)}\right).
\end{gather}

For a computation of $d_{v,k}^+$ we choose an arbitrary $(\mbC^*)^2\times(\mbC^*)^r$-fixed point $p$ in $\mcM(r,n)^{T^{\vec\theta}_{\alpha,\beta}}_{v,k}$. Let $D$ be the corresponding $r$-tuple of Young diagrams. We have
\begin{align*}
d_{v,k}^+=&\sum_{i,j}\sharp\left\{s\in D_i\left|
\begin{smallmatrix}
\theta_j-\theta_i-\alpha l_{D_j}(s)+\beta (a_{D_i}(s)+1)\equiv 0\mmod m\\
\theta_j-\theta_i-\alpha l_{D_j}(s)+\beta (a_{D_i}(s)+1)>0
\end{smallmatrix}
\right.\right\}\\
&+\sum_{i,j}\sharp\left\{s\in D_j\left|
\begin{smallmatrix}
\theta_j-\theta_i+\alpha (l_{D_i}(s)+1)-\beta a_{D_j}(s)\equiv 0\mmod m\\
\theta_j-\theta_i+\alpha (l_{D_i}(s)+1)-\beta a_{D_j}(s)>0
\end{smallmatrix}
\right.\right\}\\
=&\sum_{i,j}\sharp\left\{s\in D_i\left|
\begin{smallmatrix}
\theta_j-\theta_i-\alpha l_{D_j}(s)+\beta (a_{D_i}(s)+1)\equiv 0\mmod m\\
\theta_j-\theta_i-\alpha l_{D_j}(s)+\beta (a_{D_i}(s)+1)>0
\end{smallmatrix}
\right.\right\}\\
&+\sum_{i,j}\sharp\left\{s\in D_i\left|
\begin{smallmatrix}
\theta_j-\theta_i-\alpha l_{D_j}(s)+\beta (a_{D_i}(s)+1)\equiv 0\mmod m\\
\theta_j-\theta_i-\alpha l_{D_j}(s)+\beta (a_{D_i}(s)+1)<m
\end{smallmatrix}
\right.\right\}\\
=&\sum_{i,j}\sharp\left\{s\in D_i|
\theta_j-\theta_i-\alpha l_{D_j}(s)+\beta (a_{D_i}(s)+1)\equiv 0\mmod m\right\}\\
&+\sum_{i,j}\sharp\left\{s\in D_i\left|
\begin{smallmatrix}
\theta_j-\theta_i-\alpha l_{D_j}(s)+\beta (a_{D_i}(s)+1)\equiv 0\mmod m\\
0<\theta_j-\theta_i-\alpha l_{D_j}(s)+\beta (a_{D_i}(s)+1)<m
\end{smallmatrix}
\right.\right\}.
\end{align*}
It is easy to see that the last sum is equal to zero, thus
\begin{gather*}
d_{v,k}^+=\sum_{i,j}\sharp\left\{s\in D_i|\theta_j-\theta_i-\alpha l_{D_j}(s)+\beta (a_{D_i}(s)+1)\equiv 0\mmod m\right\}.
\end{gather*}
On the other hand
\begin{align*}
&\dim\mcM(r,n)^{\Gamma^{\vec\theta}_{\alpha,\beta}}_v=\\
=&\sum_{i,j}\sharp\left\{s\in D_i|\theta_j-\theta_i-\alpha l_{D_j}(s)+\beta (a_{D_i}(s)+1)\equiv 0\mmod m\right\}\\
&+\sum_{i,j}\sharp\left\{s\in D_j|\theta_j-\theta_i+\alpha (l_{D_i}(s)+1)-\beta a_{D_j}(s)\equiv 0\mmod m\right\}\\
=&2\sum_{i,j}\sharp\left\{s\in D_i|\theta_j-\theta_i-\alpha l_{D_j}(s)+\beta (a_{D_i}(s)+1)\equiv 0\mmod m\right\}.
\end{align*}
Hence $d_{v,k}^+=\frac{1}{2}\dim\mcM(r,n)^{\Gamma^{\vec\theta}_{\alpha,\beta}}_v=\frac{1}{2}\dim\mfM(v,w)$.
\end{proof}

From Lemmas~\ref{lemma:decomposition} and \ref{lemma:dimension} it follows that
\begin{gather*}
\left[\mcM(r,n)^{\Gamma^{\vec\theta}_{\alpha,\beta}}_v\right]=\mbL^{\frac{1}{2}\dim\mfM(v,w)}\left[\mcM(r,n)^{T^{\vec\theta}_{\alpha,\beta}}_v\right].
\end{gather*}

Using the $(\mbC^*)^2\times(\mbC^*)^r$-action it is easy to get a cell decomposition of the varieties $\mcM(r,n)^{\Gamma^{\vec\theta}_{\alpha,\beta}}_v$ and $\mcM(r,n)^{T^{\vec\theta}_{\alpha,\beta}}_{v,k}$. Therefore
\begin{align*}
&\left[\mcM(r,n)^{\Gamma^{\vec\theta}_{\alpha,\beta}}_v\right]=\left.P_q^{BM}\left(\mcM(r,n)^{\Gamma^{\vec\theta}_{\alpha,\beta}}_v\right)\right|_{q=\mbL},\\
&\left[\mcM(r,n)^{T^{\vec\theta}_{\alpha,\beta}}_{v,k}\right]=\left.P_q\left(\mcM(r,n)^{T^{\vec\theta}_{\alpha,\beta}}_{v,k}\right)\right|_{q=\mbL}.\\
\end{align*}
The theorem is proved.

\section{Proof of Theorem \ref{theorem:quasihomogeneous}}\label{section:quasihomogeneous}

In this section we prove Theorem~\ref{theorem:quasihomogeneous}. First of all, in Section~\ref{power structure} we remind the reader a notion of a power structure over the Grothendieck ring $K_0(\nu_{\mbC})$. This technique allows us to simplify some combinatorial computations. Then in Section~\ref{subsection:cores and quotients} we review standard combinatorial constructions related to Young diagrams. In Section~\ref{subsection:Hilbert and quivers} we review a connection between Hilbert schemes and quiver varieties and do an important step in the proof of Theorem~\ref{theorem:quasihomogeneous}. Instead of considering the $T_{\alpha,\beta}$-fixed point set in the Hilbert scheme $(\mbC^2)^{[n]}$, we first look at the fixed point set of a finite subgroup of $T_{\alpha,\beta}$. Finally, in Section~\ref{subsection:proof of theorem} we combine everything and prove the theorem.   

\subsection{Power structure over $K_0(\nu_{\mbC})$}\label{power structure}

In \cite{Gusein-Zade2} there was defined a notion of a power structure over a ring and there was described a natural power structure over the Grothendieck ring $K_0(\nu_\mbC)$. This means that for a series $A(t)=1+a_1t+a_2t^2+\ldots\in 1 + t\cdot K_0(\nu_\mbC)[[t]]$ and for an element $m\in K_0(\nu_\mbC)$ one defines a series $(A(t))^m \in 1 + t\cdot K_0(\nu_\mbC)[[t]]$ so that all the usual properties of the exponential function hold. We also need the following property of this power structure. For any $i\ge 1$ and $j\ge 0$ we have (see e.g.\cite{Gusein-Zade2})
\begin{gather}\label{Lpower}
(1-\mbL^jt^i)^\mbL=1-\mbL^{j+1}t^i.
\end{gather}

\subsection{Cores and quotients}\label{subsection:cores and quotients}

In this section we review the well known construction of an $m$-core and an $m$-quotient of a Young diagram.

The set $Core_m$ is defined as the set of Young diagrams $Y$ such that for any box $s\in Y$ we have $l_Y(s)+a_Y(s)+1\not\equiv 0\mmod m$. For a Young diagram $Y$ let 
$$
w_i(Y)=\sharp\{(p,q)\in Y|p+q\equiv i\mmod m\}.
$$
We remind the reader that we consider a Young diagram as a subset of $\mbZ_{\ge 0}^2$. Let 
$$
\Pi^{m-1}=\left\{\lambda=(\lambda_0,\lambda_1,\ldots,\lambda_{m-1})\in\mbZ^m\left|\sum_{k=0}^{m-1}\lambda_k=0\right.\right\}.
$$
Define the map $\Psi\colon Core_m\to \Pi^{m-1}$ by 
$$
Core_m\ni Y\mapsto (\lambda_0,\lambda_1,\ldots,\lambda_{m-1}), \lambda_i=w_{i+1}(Y)-w_i(Y).
$$ 
The map $\Psi$ is a bijection (see e.g.\cite{JK},Ch.2.7). 

There is also a bijection (see e.g.\cite{JK},Ch.2.7.)
$$
\Phi\colon\mathcal Y\to Core_m\times \mathcal Y^m, \Phi(Y)=(\Phi(Y)_0,\Phi(Y)_1,\ldots,\Phi(Y)_m).
$$ 
We don't give a construction of this map, we will only list all necessary properties. The diagram $\Phi(Y)_0$ is called the $m$-core of the diagram $Y$ and the $m$-tuple $(\Phi(Y)_1,\Phi(Y)_2,\ldots,\Phi(Y)_m)$ is called the $m$-quotient. The bijection $\Phi$ has the following properties (see e.g.\cite{JK},Ch.2.7.):
\begin{align}
& |Y|=|\Phi(Y)_0|+m\sum_{i=1}^m|\Phi(Y)_i|;\label{first property}\\
& w_i(Y)=w_i(\Phi(Y)_0)+\sum_{i=1}^m|\Phi(Y)_i|;\label{second property}\\
& \sharp\{s\in Y|l_Y(s)+a_Y(s)+1\equiv 0\mmod m\}=\sum_{i=1}^m|\Phi(Y)_i|.\label{third property}
\end{align} 

\subsection{Hilbert schemes and quiver varieties}\label{subsection:Hilbert and quivers}

For an ideal $I\subset\mbC[x,y]$ of codimension $n$ let $V(I)=\mbC[x,y]/I$ and $B_1,B_2\in GL(V(I))$ be the operators of the multiplications by $x$ and $y$ correspondingly. Let $i\colon\mbC\to V(I)$ be the linear map that sends $1\in\mbC$ to the unit in $\mbC[x,y]$. Define the map $f\colon (\mbC^2)^{[n]}\to\mcM(1,n)$ by $I\mapsto [(B_1,B_2,i,0)]$. This map is an isomorphism (see e.g.\cite{Nakajima}).

For integers $\mu$ and $\nu$ let $\Gamma_{\nu,\mu}$ be the finite subgroup of $(\mbC^*)^2$ defined by 
$$
\Gamma_{\nu,\mu}=\left\{(\zeta^{j\nu},\zeta^{j\mu})\in(\mbC^*)^2\left|\zeta=\exp\left(\frac{2\pi i}{m}\right)\right.\right\}.
$$
It is clear that the isomorphism $f$ transforms the $T_{\alpha,\beta}$-action on $(\mbC^2)^{[n]}$ to the $T^{\vec 0}_{\alpha,\beta}$-action on $\mcM(1,n)$ and the $\Gamma_{\alpha,\beta}$-action to the $\Gamma^{\vec 0}_{\alpha,\beta}$-action. Thus, by Lemma~\ref{lemma:finite group}, we have
\begin{align}
&\left((\mbC^2)^{[n]}\right)^{\Gamma_{\alpha,\beta}}=\coprod_{\substack{v\in\mbZ_{\ge 0}^m\\\sum v_i=n}}\mfM(v,e_0),\notag\\
&\left((\mbC^2)^{[n]}\right)^{T_{\alpha,\beta}}=\coprod_{\substack{v\in\mbZ_{\ge 0}^m\\\sum v_i=n}}\mfM(v,e_0)^{\widetilde T_{\alpha,\beta}},\label{formula:finite group2}
\end{align}
where by $e_0$ we denote the vector $(1,0,\ldots,0)\in\mbZ_{\ge 0}^m$. Until the end of this section we consider a quiver variety $\mfM(v,e_0)$ as a subset of $\mcM(1,\sum v_i)=(\mbC^2)^{[\sum v_i]}$.

The last factor $\mbC^*$ of the product $(\mbC^*)^2\times\mbC^*$ acts trivially on $\mcM(1,n)$, so now we start to consider only the $(\mbC^*)^2$-action on $\mcM(1,n)$. 

\subsection{Proof of Theorem \ref{theorem:quasihomogeneous}}\label{subsection:proof of theorem}
For a vector $v\in\mbZ_{\ge 0}^m$ let $|v|=\sum_{i=0}^{m-1}v_i$. By~\eqref{formula:finite group2} and Theorem~\ref{firsttheorem}, we have
$$
\sum_{n\ge 0}P_q\left(\left((\mbC^2)^{[n]}\right)^{T_{\alpha,\beta}}\right)t^n=\sum_{v\in\mbZ^m_{\ge 0}}q^{-\frac{1}{2}\dim\mfM(v,e_0)}P^{BM}_q\left(\mfM(v,e_0)\right)t^{|v|}.
$$
If the variety $\mfM(v,e_0)$ is nonempty, then (see e.g.\cite{Nakajima2})
\begin{gather}\label{dimension formula}
\dim\mfM(v,e_0)=2v_0-\sum_{i=0}^{m-1}(v_i-v_{i+1})^2.
\end{gather}
Here we follow the convention $v_m=v_0$. For $\lambda\in\Pi^{m-1}$ let
\begin{align*}
&v_0(\lambda)=\frac{1}{2}\sum_{k=0}^{m-1}\lambda_k^2,\\
&n(\lambda)=mv_0(\lambda)+\sum_{k=0}^{m-2}(m-1-k)\lambda_k.
\end{align*}
Using these notations and formula \eqref{dimension formula} we get
\begin{align*}
&\sum_{v\in\mbZ^m_{\ge 0}}q^{-\frac{1}{2}\dim\mfM(v,e_0)}P^{BM}_q\left(\mfM(v,e_0)\right)t^{|v|}=\\
&=\sum_{\lambda\in\Pi^{m-1}}t^{n(\lambda)}\sum_{\substack{v\in\mbZ^m_{\ge 0}\\v_{i+1}-v_i=\lambda_i}}P^{BM}_q\left(\mfM(v,e_0)\right)\left(q^{-\frac{1}{m}}t\right)^{m(v_0-v_0(\lambda))}.
\end{align*}

\begin{lemma}\label{important lemma}
For any $\lambda\in\Pi^{m-1}$ we have
\begin{gather*}
\sum_{\substack{v\in\Z^m_{\ge 0}\\v_{i+1}-v_i=\lambda_i}}P^{BM}_q\left(\mfM(v,e_0)\right)t^{|v|}=\frac{t^{n(\lambda)}}{\prod\limits_{i\ge 1}(1-q^it^{mi})^{m-1}(1-q^{i+1}t^{mi})}.
\end{gather*}
\end{lemma}
Before a proof of this lemma we introduce a new notation and prove two useful lemmas. 

In the proof of Lemma~\ref{lemma:decomposition} we used the morphism $\pi\colon\mcM(r,n)\to\mcM_0(r,n)$. We have $\mcM_0(1,n)=S^n(\mbC^2)$ (see e.g.\cite{Nakajima}). Slightly changing notations we denote now by $\pi$ the morphism $\mcM(1,n)\to S^n(\mbC^2)$. It can be described explicitly as follows. Let $[(B_1,B_2,i,j)]\in\mcM(1,n)$. We can make $B_1$ and $B_2$ simultaneously into upper triangular matrices with numbers $\lambda_i$ and $\mu_i$ on the diagonals. The morphism $\pi$ is given by $\pi(B_1,B_2,i,j)=\{(\lambda_1,\mu_1),\ldots,(\lambda_n,\mu_n)\}$ (see e.g.\cite{Nakajima}).

It is useful to note that the subgroups $\Gamma_{\alpha,\beta}$ and $\Gamma_{1,-1}$ of $(\mbC^*)^2$ coincide, therefore
$$
\mcM(1,n)^{\Gamma_{\alpha,\beta}}=\mcM(1,n)^{\Gamma_{1,-1}}.
$$
For any $\Gamma_{1,-1}$-invariant subset $Z\subset\mbC^2$ and any vector $\lambda\in\Pi^{m-1}$ let
$$
H_{Z,\lambda}(t)=\sum_{\substack{v\in\Z^m_{\ge 0}\\v_{i+1}-v_i=\lambda_i}}\left[\mfM(v,e_0)\cap\pi^{-1}(S^{|v|}Z)\right]t^{|v|}.
$$ 
We denote by $\mbC_x$ the $x$-axis in the plane $\mbC^2$.

\begin{lemma}\label{lemma:middle step}
For any $\lambda\in\Pi^{m-1}$ consider the unique diagram $Y_\lambda\in Core_m$ such that $\Psi(Y_\lambda)=\lambda$. Then we have
\begin{gather}\label{middle step}
H_{\mbC_x,\lambda}(t)=\frac{t^{|Y_\lambda|}}{\prod\limits_{i\ge 1}(1-\mbL^it^{mi})^m}.
\end{gather}
\end{lemma}
\begin{proof}
The set of fixed points of the $(\mbC^*)^2$-action on $\mcM(1,n)$ is parametrized by the set of Young diagrams $Y$ such that $|Y|=n$. Let $p$ be the fixed point corresponding to a Young diagram $Y$, then, by \eqref{weight decomposition}, we have
\begin{gather}\label{small weight decomposition}
T_p(\mcM(1,n))=\sum_{s\in Y}\left(t_1^{-l_Y(s)}t_2^{a_Y(s)+1}+t_1^{l_Y(s)+1}t_2^{-a_Y(s)}\right).
\end{gather}
We choose $\gamma\gg 1$ and for each point $p\in\mfM(v,e_0)^{(\mbC^*)^2}$ we define the attracting set $C_p$ as follows 
$$
C_p=\{z\in\mfM(v,e_0)|\lim\nolimits_{t\to 0,t\in T_{1,-\gamma}}tz=p\}.
$$
Clearly, if $z\in \mbC_x$, then $\lim_{t\to 0,t\in T_{1,-\gamma}}tz=0$, and if $z\in\mbC^2\backslash\mbC_x$, then $tz$ goes to infinity. By \cite{B1,B2}, the sets $C_p$ form a cell decomposition of $\mfM(v,e_0)\cap\pi^{-1}(S^{|v|}\mbC_x)$. Using \eqref{small weight decomposition} we obtain
\begin{gather}\label{combinatorics}
\left[\mfM(v,e_0)\cap\pi^{-1}(S^{|v|}\mbC_x)\right]=\sum_{\substack{Y\in\mathcal Y\\w_i(Y)=v_i}}\mbL^{\sharp\{s\in Y|l_Y(s)+a_Y(s)+1\equiv 0\mmod m\}}.
\end{gather}
The formula \eqref{middle step} follows from \eqref{combinatorics} and properties \eqref{first property},\eqref{second property} and \eqref{third property}.
\end{proof}

\begin{lemma}\label{core lemma}
For any $Y\in Core_m$ we have $|Y|=n(\Psi(Y))$.
\end{lemma}
\begin{proof}
Consider the quiver variety $\mfM(w(Y),e_0)$. From the properties of the bijection $\Phi$ it follows that if $Y'$ is a Young diagram such that $|Y'|=|Y|$ and $w(Y)=w(Y')$, then $Y'=Y$. Thus, the $(\mbC^*)^2$-fixed point set in $\mfM(w(Y),e_0)$ consists of only one point. Using the Bialynicki-Birula theorem we can construct a cell decomposition of $\mfM(w(Y),e_0)$ and it is easy to see that the unique cell has dimension~$0$. Therefore, $\mfM(w(Y),e_0)$ is just a point. By~\eqref{dimension formula}, $w_0(Y)=v_0(\Psi(Y))$ and clearly $|Y|=n(\Psi(Y))$.  
\end{proof}

\begin{proof}[Proof of Lemma  \ref{important lemma}]
For $\lambda=\vec 0$ this lemma was proved in \cite{Gusein-Zade}. 

Since $[\mfM(v,e_0)]=\left.P^{BM}_q(\mfM(v,e_0))\right|_{q=\mbL}$, it is sufficient to prove that
\begin{gather*}
\sum_{\substack{v\in\Z^m_{\ge 0}\\v_{i+1}-v_i=\lambda_i}}\left[\mfM(v,e_0)\right]t^{|v|}=\frac{t^{n(\lambda)}}{\prod\limits_{i\ge 1}(1-\mbL^it^{mi})^{m-1}(1-\mbL^{i+1}t^{mi})}.
\end{gather*}

The $\Gamma_{1,-1}$-action on $\mbC^2\backslash \mbC_x$ is free. Therefore, if the intersection of $\mfM(v,e_0)$ with $\pi^{-1}(S^{|v|}(\mbC^2\backslash\mbC_x))$ is nonempty, then $v_0=v_1=\ldots=v_{m-1}$. We get 
\begin{gather}\label{product formula}
H_{\mbC^2,\lambda}(t)=H_{\mbC_x,\lambda}(t)H_{\mbC^2\backslash\mbC_x,\vec 0}(t).
\end{gather}
We denote by $O$ the origin of $\mbC^2$. Let
$$
H_O(t)=\sum_{n\ge 0}[\mcM(1,n)\cap\pi^{-1}(S^n(O))]t^n.
$$
From \cite{Gusein-Zade} (see Theorem 1) it follows that  
\begin{gather*}
H_{\mbC^2\backslash\mbC_x,\vec 0}(t)=H_O(t^m)^{[(\mbC^2\backslash\mbC_x)/\Gamma_{1,-1}]}.
\end{gather*}
It is easy to check that $[(\mbC^2\backslash\mbC_x)/\Gamma_{1,-1}]=\mbL^2-\mbL$. Therefore we have
\begin{gather}
H_{\mbC^2\backslash\mbC_x,\vec 0}(t)=\left(\prod_{i\ge 1}\frac{1}{(1-\mbL^{i-1}t^{mi})}\right)^{\mbL^2-\mbL}\stackrel{\text{by \eqref{Lpower}}}{=}\prod_{i\ge 1}\frac{1-\mbL^it^{mi}}{1-\mbL^{i+1}t^{mi}}.\label{gusein-zade formula}
\end{gather}

If we combine formulas \eqref{middle step}, \eqref{product formula} and \eqref{gusein-zade formula} and also Lemma~\ref{core lemma}, we get the proof of the lemma.
\end{proof}

Using Lemma~\ref{important lemma} we get
\begin{align*}
&\sum_{\lambda\in\Pi^{m-1}}t^{n(\lambda)}\sum_{\substack{v\in\mbZ^m_{\ge 0}\\v_{i+1}-v_i=\lambda_i}}P^{BM}_q\left(\mfM(v,e_0)\right)\left(q^{-\frac{1}{m}}t\right)^{m(v_0-v_0(\lambda))}=\\
&=\left(\prod\limits_{i\ge 1}\frac{1}{(1-t^{mi})^{m-1}(1-qt^{mi})}\right)\left(\sum_{\lambda\in\Pi^{m-1}}t^{n(\lambda)}\right).
\end{align*}
By Lemma \ref{core lemma}, $\sum_{\lambda\in\Pi^{m-1}}t^{n(\lambda)}=\sum_{Y\in Core_m}t^{|Y|}$. We have (see e.g.\cite{Gusein-Zade})
$$
\sum_{Y\in Core_m}t^{|Y|}=\prod_{i\ge 1}\frac{(1-t^{mi})^m}{(1-t^i)}.
$$
This completes the proof of the theorem.

\section{Proofs of Theorems \ref{first identity} and \ref{second identity}}\label{section:combinatorial identities}

Here we prove two combinatorial identities from Section \ref{subsection:combinatorial identities}.

\subsection{Proof of Theorem \ref{first identity}}

Consider the $(\mbC^*)^2$-action on $\left((\mbC^2)^{[n]}\right)^{T_{\alpha,\beta}}$. Let $p\in(\mbC^2)^{[n]}$ be the fixed point corresponding to a Young diagram~$Y$. By \eqref{weight decomposition}, the weight decomposition of $T_p(((\mbC^2)^{[n]})^{T_{\alpha,\beta}})$ is given by 
\begin{align}
&T_p(((\mbC^2)^{[n]})^{T_{\alpha,\beta}})=\notag\\
&=\sum_{\substack{s\in Y\\ \alpha(l_Y(s)+1)=\beta a_Y(s)}}t_1^{l_Y(s)+1}t_2^{-a_Y(s)}+\sum_{\substack{s\in Y\\ \alpha l_Y(s)=\beta(a_Y(s)+1)}}t_1^{-l_Y(s)}t_2^{a_Y(s)+1}.\label{alpha,beta decomposition}
\end{align}
Let $\gamma$ be a big positive integer $\gamma$. By \cite{B1,B2}, the variety $\left((\mbC^2)^{[n]}\right)^{T_{\alpha,\beta}}$ has a cellular decomposition with the cells $C_p=\left\{\left.z\in\left((\mbC^2)^{[n]}\right)^{T_{\alpha,\beta}}\right|\lim_{t\to 0,t\in T_{1,\gamma}}tz=p\right\}$. By \eqref{alpha,beta decomposition}, we have $\dim C_p=\sharp\{s\in Y|\alpha l_Y(s)=\beta(a_Y(s)+1)\}$. Thus, we have
\begin{gather*}
\sum_{n\ge 0}P_q(((\mbC^2)^{[n]})^{T_{\alpha,\beta}})t^n=\sum_{Y\in\mcY}q^{\sharp\{s\in Y|\alpha l_Y(s)=\beta(a_Y(s)+1)\}}t^{|Y|}.
\end{gather*}
Now Theorem \ref{first identity} follows from Theorem \ref{theorem:quasihomogeneous}.

\subsection{Proof of Theorem \ref{second identity}}
In \cite{Iarrobino} it is proved that the set of irreducible components of the variety $((\mbC^2)^{[n]})^{T_{1,1}}$ is parametrized by partitions $\lambda$ such that $\frac{\lambda_1(\lambda_1-1)}{2}+|\lambda|=n$. The Poincar\'e polynomial of the irreducible component corresponding to a partition $\lambda$ is equal to~(see~\cite{Iarrobino})
$$
\prod_{i\ge 1}\genfrac[]{0pt}{}{\lambda_i-\lambda_{i+2}+1}{\lambda_{i+1}-\lambda_{i+2}}_q.
$$
Combining this fact with Theorem \ref{theorem:quasihomogeneous} we get the proof of Theorem~\ref{second identity}.

\section{Proof of Theorem \ref{theorem:irreducible components}}\label{section:irreducible components}

Let $\alpha+\beta=m$. Similar to Lemma \ref{lemma:finite group} we have the decomposition
\begin{gather}
\mcM(r,n)^{\Gamma_{\alpha,\beta}^{\vec\omega}}=\coprod_{\substack{v\in\mbZ_{\ge 0}^m\\|v|=n}}\mfM(v,\vec\mu),\label{formula:decom}
\end{gather}
and the $T^{\vec\omega}_{\alpha,\beta}$-action on the left-hand side of \eqref{formula:decom} corresponds to the $\widetilde T_{\alpha,\beta}$-action on the right-hand side. Using Theorem~\ref{firsttheorem} we get
$$
\sum_{n\ge 0}h_0\left(\mcM(r,n)^{T_{\alpha,\beta}^{\vec\omega}}\right)q^n=\sum_{v\in\mbZ^m_{\ge 0}}\dim H^{BM}_{\frac{1}{2}\dim\mfM(v,\vec\mu)}(\mfM(v,\vec\mu))q^{|v|}.
$$ 
In \cite{Nakajima1} it is proved that the space $\bigoplus_{v\in\mbZ^m_{\ge 0}}H^{BM}_{\frac{1}{2}\dim\mfM(v,\vec\mu)}(\mfM(v,\vec\mu))$ is an irreducible highest weight representation of $\widehat{sl}_m$ with the highest weight $\vec\mu$. This completes the proof of Theorem~\ref{theorem:irreducible components}.


\begin{thebibliography}{99}

\bibitem{B1} A. Bialynicki-Birula. Some theorems on actions of algebraic groups. Ann. Math. 98 (1973) 480-497.

\bibitem{B2} A. Bialynicki-Birula. Some properties of the decompositions of algebraic varieties determined by actions of a torus. Bull. Acad. Pol. Sci. Ser. Sci. Math. astron. Phys. 24 (1976) 667-674.

\bibitem{Buryak} A. Buryak. The classes of the quasihomogeneous Hilbert schemes of points on the plane. Moscow Mathematical Journal. Volume 12 (2012) Number 1, 1-17.

\bibitem{BuryakFeigin} A. Buryak, B. L. Feigin. Homogeneous components in the moduli space of sheaves and Virasoro characters. Journal of Geometry and Physics. Volume 62, Issue 7, July 2012, 1652-1664.

\bibitem{Costello} K. Costello and I. Grojnowski. Hilbert schemes, Hecke algebras and the Calogero-Sutherland system. math.AG/0310189.

\bibitem{Ellingsrud} G. Ellingsrud, S. A. Stromme. On the homology of the Hilbert scheme of points in the plane. Invent. math. 87, 343-352 (1987).

\bibitem{Evain} L. Evain. Irreducible components of the equivariant punctual Hilbert schemes.  Adv. Math.  185  (2004),  no. 2, 328-346.

\bibitem{FFJMM} B. Feigin; E. Feigin, M. Jimbo, T. Miwa, E. Mukhin. Quantum continuous $gl_{\infty}$: tensor products of Fock modules and $W_n$-characters. Kyoto J. Math. 51 (2011), no. 2, 365-392.

\bibitem{Gusein-Zade} S. M. Gusein-Zade, I. Luengo, A. Melle Hernandez. On generating series of classes of equivariant Hilbert schemes of fat points. Mosc. Math. J. 10 (2010), no. 3.

\bibitem{Gusein-Zade2} S.M. Gusein-Zade, I. Luengo, A. Melle-Hern´andez. A power structure over the Grothendieck ring of varieties. Math. Res. Lett. 11 (2004), no.1, 49-57.

\bibitem{Gottsche} L. Gottsche. Hilbert schemes of points on surfaces. ICM Proceedings, Vol. II (Beijing, 2002), 483-494.

\bibitem{Iarrobino} A. Iarrobino, J. Yameogo. The family $G_T$ of graded artinian quotients of $k[x,y]$ of given Hilbert function. Special issue in honor of Steven L. Kleiman.  Comm. Algebra  31  (2003),  no. 8, 3863-3916.

\bibitem{JK} G. James, A. Kerber. The representation theory of the symmetric group. Encyclopedia of Mathematics and its Applications, vol. 16, Addison-Wesley Publishing Co., Reading, Mass., 1981.

\bibitem{Jimbo} M. Jimbo, K. C. Misra, T. Miwa, M. Okado. Combinatorics of representations of $U_q(\widehat{sl}_n)$ at $q=0$. Comm. Math. Phys. 136 (1991), no. 3, 543-566.

\bibitem{Lehn1} M. Lehn. Chern classes of tautological sheaves on Hilbert schemes of points on surfaces. Invent. Math. 136 (1999), no. 1, 157-207.

\bibitem{Lehn2} M. Lehn and C. Sorger. Symmetric groups and the cup product on the cohomology of Hilbert schemes. Duke Math. J. 110 (2001), no. 2, 345-357.

\bibitem{Looijenga} E. Looijenga. Motivic measures. Seminaire Bourbaki, 1999-2000(874), 2000.

\bibitem{Lusztig} G. Lusztig. On quiver varieties. Adv. in Math. 136 (1998) 141-182.

\bibitem{Li} W.-P. Li, Z. Qin, W. Wang, Vertex algebras and the cohomology ring structure of Hilbert schemes of points on surfaces. Math. Ann. 324
(2002), no. 1, 105-133.

\bibitem{Nakajima1} H. Nakajima. Instantons on ALE spaces, quiver varieties, and Kac-Moody algebras, Duke Math. J. 76 (1994), 365-416.

\bibitem{Nakajima} H. Nakajima. Lectures on Hilbert schemes of points on surfaces. AMS, Providence, RI, 1999.

\bibitem{Nakajima3} H. Nakajima. Quiver varieties and Kac-Moody algebras. Duke Math. J. 91 (1998) 515-560.

\bibitem{Nakajima2} H. Nakajima, K. Yoshioka. Lectures on instanton counting. Algebraic structures and moduli spaces, 31–101, CRM Proc. Lecture Notes, 38, Amer. Math. Soc., Providence, RI, 2004.

\bibitem{Vasserot} E. Vasserot, Sur l'anneau de cohomologie du schema de Hilbert de $\mbC^2$, C. R. Acad. Sci. Paris S´er. I Math. 332 (2001), no. 1, 7-12.

\end{thebibliography}
\end{document}